\documentclass{amsart}
\usepackage[T1]{fontenc}
\usepackage{verbatim}
\usepackage{tgpagella}
\usepackage{xcolor}
\usepackage{hyperref}
\usepackage{cleveref}
\usepackage{comment}
\usepackage{amsmath,amssymb,amsthm,mathtools,mathrsfs,amsfonts}

\usepackage{tikz}
\usetikzlibrary{calc,matrix,decorations.pathreplacing}
\pgfkeys{tikz/mymatrixenv/.style={decoration={brace},every left delimiter/.style={xshift=8pt},every right delimiter/.style={xshift=-8pt}}}
\pgfkeys{tikz/mymatrix/.style={matrix of math nodes,nodes in empty cells,left delimiter={[},right delimiter={]},inner sep=1pt,outer sep=1.5pt,column sep=2pt,row sep=2pt,nodes={minimum width=20pt,minimum height=10pt,anchor=center,inner sep=0pt,outer sep=0pt}}}
\pgfkeys{tikz/mymatrixbrace/.style={decorate,thick}}

\newcommand*\mymatrixbraceleft[4][m]{
    \draw[mymatrixbrace] (#1.east|-#1-#2-1.north east) -- node[right=2pt] {#4} (#1.east|-#1-#3-1.south east);
}
\newcommand*\mymatrixbracetop[4][m]{
    \draw[mymatrixbrace] (#1.north-|#1-1-#2.north west) -- node[above=2pt] {#4} (#1.north-|#1-1-#3.north east);
}

\def\C{\mathcal{C}}
\def\I{\mathcal{I}}
\def\J{\mathcal{J}}
\def\K{\mathcal{K}}
\def\KK{\mathbb{K}}
\def\V{\mathcal{V}}

\def\Re{\mathbb{R}}
\DeclareMathOperator{\codim}{codim}
\DeclareMathOperator{\rank}{rank}

\newtheorem{theorem}{Theorem}
\newtheorem{example}[theorem]{Example}
\newtheorem{lemma}[theorem]{Lemma}
\newtheorem{definition}[theorem]{Definition}
\newtheorem{question}[theorem]{Question}
\newtheorem{conjecture}[theorem]{Conjecture}
\newtheorem{proposition}[theorem]{Proposition}

\newcommand{\hconcat}{\begin{bmatrix} U & \begin{matrix} X_1+Y_1 & X_2 & 0\\ 0 & Y_4&  X_{t+1:n,1:d} \end{matrix} \end{bmatrix}}
\newcommand{\vconcat}{\begin{bmatrix} 
U_{1:t,1:d} &  X_1+Y_1 & X_2 \\ 
U_{t+1:n,1:d} & 0 & Y_4 \\ 
0 & X_3 & X_4 
\end{bmatrix}}
\begin{document}

\title{Toeplitz Unlabeled Sensing: Algebraic Theory}

\author{Xin Hong \, \, \, \, \, \, Manolis C. Tsakiris}
\thanks{Part of the results in this paper have been announced in the short expository paper \cite{ICASSP26}.}
\thanks{This research is partially supported by the National Key R$\&$D Program of China (2023YFA1009402) and by the CAS Project for Young Scientists in Basic Research (YSBR-
034).
}
\address{State Key Laboratory of Mathematical Sciences \\ Academy of Mathematics and Systems Science \\ Chinese Academy of Sciences \\ 100109, Beijing, hongxin@amss.ac.cn, manolis@amss.ac.cn}

\begin{abstract}
Unlabeled sensing is the inverse problem of recovering an element of a vector subspace of $\mathbb{R}^n$, from its image under an unknown permutation of the coordinates and knowledge of the subspace. Here we study this problem from an algebraic point of view for the special class of subspaces that admit a Toeplitz basis. 
\end{abstract}

\maketitle

%%%%%%%%
\section{Introduction}

Let $\V$ be a $d$-dimensional linear subspace of $\mathbb{R}^n$, let $v \in \V$, $\pi: \mathbb{R}^n \rightarrow \mathbb{R}^n$ a permutation of the coordinates of $\mathbb{R}^n$, and $\rho: \mathbb{R}^n \rightarrow \mathbb{R}^n$ a coordinate projection (viewed as an endomorphism). In their seminal work \cite{Unnikrishnan-Allerton2015,Unnikrishnan-TIT18} the authors posed and studied the \emph{unlabeled sensing} problem, which concerns the recovery of $v$ from the data $\rho \pi(v), \, \V$. Their main result was that if $2 d \le r$, where $r$ is the number of coordinates preserved by the coordinate projection $\rho$, and if $\V$ is sufficiently generic, then unique recovery is possible. This result was generalized in \cite{tsakiris2018eigenspace,Peng-ACHA-21} as \emph{homomorphic sensing}, where the authors replaced $\rho \pi$ by an arbitrary endomorphism of $\mathbb{R}^n$, obtaining analogous statements valid for a generic linear subspace $\V$. 

On the other hand, the subspaces $\V$ that occur in applications often have special structure, as dictated by the nature of the application. In this paper we are motivated by the potential application of unlabeled sensing to signal processing and control systems, where, in a system identification scenario, the output signal of some linear time-invariant filter is only available up to a permutation of its samples. Viewed as an element of $\Re^n$ (in a finite interval in discrete time), the output signal $v$ is constrained to lie in the column-space $\V$ of a Toeplitz matrix, the rows of which correspond to input regressor vectors \cite{Oppenheim}. We note that the theorem of \cite{Unnikrishnan-TIT18} can not be applied directly, because the set of $d$-dimensional linear subspaces of $\Re^n$ that can be represented by an $n \times d$ Toeplitz matrix is an $(n+d-2)$-dimensional subvariety of the $d(n-d)$-dimensional Grassmannian variety $\operatorname{Gr}(d,n)$ \cite[Remark 4.15]{HankelPlanes}. 
We demonstrate this important point with an example. 

\begin{example} \label{ex:6x3}
Consider the following cyclic permutation $\Psi$ and matrix of variables $A$, with $n=6$ and $d=3$:

$$\Psi = \begin{bmatrix}
0 & 1 & 0 & 0 & 0 & 0 \\
0 & 0 & 1 & 0 & 0 & 0 \\
0 & 0 & 0 & 1 & 0 & 0 \\
0 & 0 & 0 & 0 & 1 & 0 \\
0 & 0 & 0 & 0 & 0 & 1 \\
1 & 0 & 0 & 0 & 0 & 0\\
\end{bmatrix}, \, \, \, A = \begin{bmatrix}
a_{11} & a_{12} & a_{13} \\
a_{21} & a_{22} & a_{23} \\
a_{31} & a_{32} & a_{33} \\
a_{41} & a_{42} & a_{43} \\
a_{51} & a_{52} & a_{53} \\
a_{61} & a_{62} & a_{63} \\
\end{bmatrix}. $$

According to \cite{Unnikrishnan-TIT18}, there are no $3$-dimensional vectors $\xi_1 \neq \xi_2$ with entries rational functions in the variables $\{a_{ij}\}$ such that $A \xi_1 = \Psi A \xi_2$. Instead, consider now the $6 \times 3$ Toeplitz matrix of variables 
\begin{align*}
A = \begin{bmatrix}
a_{0} & a_{-1} & a_{-2} \\
a_{1} & a_{0} & a_{-1} \\
a_{2} & a_{1} & a_{0} \\
a_{3} & a_{2} & a_{1} \\
a_{4} & a_{3} & a_{2} \\
a_{5} & a_{4} & a_{3} \\
\end{bmatrix}. \label{eq:Toeplitz-variables}
\end{align*} With $\beta := a_{-2}-a_4$, and $\gamma := a_5 - a_{-1},$
we have the identity
\begin{align*}
\underbrace{\begin{bmatrix}
a_{0} & a_{-1} & a_{-2} \\
a_{1} & a_{0} & a_{-1} \\
a_{2} & a_{1} & a_{0} \\
a_{3} & a_{2} & a_{1} \\
a_{4} & a_{3} & a_{2} \\
a_{5} & a_{4} & a_{3} \\
\end{bmatrix}}_{A}
\underbrace{\begin{bmatrix}
\beta \\
\gamma \\
0
\end{bmatrix}}_{\xi_1} = 
\underbrace{\begin{bmatrix}
a_{1} & a_{0} & a_{-1} \\
a_{2} & a_{1} & a_{0} \\
a_{3} & a_{2} & a_{1} \\
a_{4} & a_{3} & a_{2} \\
a_{5} & a_{4} & a_{3} \\
a_{0} & a_{-1} & a_{-2} 
\end{bmatrix}}_{\Psi A}
\underbrace{\begin{bmatrix}
0 \\
\beta \\
\gamma 
\end{bmatrix}}_{\xi_2}. 
\end{align*} Hence for any permutations $\Pi_1$ and $\Pi_2$ such that $\Pi_1^{-1} \Pi_2 = \Psi$, it is impossible to tell from the "measurement vector" $\Pi_1 A \xi_1$ and knowledge of the subspace $\V$ if the point $v \in \V$ we are looking for is $A\xi_1$ or $A \xi_2$. 
\end{example}

Example \ref{ex:6x3} highlights the need to establish theory for unlabeled sensing that explicitly takes into consideration the combinatorics imposed by the Toeplitz structure, and we set this as the principal objective of this paper. In particular, we study the following question:

\begin{question} \label{question:US}
Suppose that $n \ge 2d$. Let $\KK$ be an infinite field and $\V$ a $d$-dimensional linear subspace of $\KK^n$ spanned by the columns of a Toeplitz matrix $V \in \KK^{n \times d}$. Let $\pi: \KK^n \rightarrow \KK^n$ be a permutation of the coordinates of $\KK^n$. Suppose $v_1, \, v_2 \in \V$ such that $v_1 = \pi(v_2)$. Under what conditions on $V$ and $\pi$ can we always conclude $v_1 = v_2$? 
\end{question}

It will be convenient to make the following definition.

\begin{definition} \label{definition:US}
If the conclusion $v_1 = v_2$ holds in Question \ref{question:US}, then we say that the Unlabeled Sensing Property holds true for $\V$ and $\pi$, or for short, that $\operatorname{USP}(\V,\pi)$ holds.
\end{definition}

We are able to provide a partial answer to Question \ref{question:US}, and to state it we need some notation. We let $\C(V)$ be the column-space of $V$, and set $r_0:= \rank(I-\Pi)$, where $\Pi$ is the $n \times n$ matrix that represents the permutation $\pi$. We denote by $J$ the $n \times n$ lower Jordan block associated to eigenvalue zero, and by convention $J^0 = I$. If $A$ is a matrix, $A_{\alpha:\alpha', \beta:\beta'}$ is the submatrix of $A$ corresponding to rows $\alpha, \, \alpha+1, \cdots, \alpha'$ and columns $\beta, \, \beta+1, \cdots, \beta'$. We have:

\begin{theorem} \label{thm:US}
Suppose $r_0\leq d$, then $\operatorname{USP}\big(\C(V),\pi\big)$ holds for $V$ generic Toeplitz. Suppose $r_0> d$, if there exists an integer $0\leq t \leq \frac{d}{2}$ such that
    \[ r_{t} :=  \rank \big( (\Pi - J^t)_{t+1:n,1:n} \big) \leq d - 2t,  \]
    or if there exists an integer $-\frac{d}{2} \le t < 0$ such that
    \[ r_{t} :=  \rank \big( (\Pi - (J^\top)^{-t})_{1:n+t,1:n} \big) \leq d + 2t,  \]then $\operatorname{USP}\big(\C(V),\pi\big)$ holds for $V$ generic Toeplitz if and only if $r_t=d-2|t|$.
\end{theorem} 

In Theorem \ref{thm:US}, but also in subsequent statements, the attribute "generic Toeplitz" is to be interpreted in the sense that inside the affine space $\mathbb{T}_{n \times d}$ of $n \times d$ Toeplitz matrices with entries in $\KK$, there is a dense open subset $\mathscr{U}$ in the Zariski topology (\cite{harris2013algebraic}) of $\mathbb{T}_{n \times d}$, such that for every 
Toeplitz matrix $V \in \mathscr{U}$, the statement of the theorem holds true. That this is so, becomes evident from the proof of the theorem, which will be delivered in \S \ref{subsection:Proof-thm:US}.

While a complete answer to Question \ref{question:US} remains elusive, we pose the following conjecture, which we have verified by exhaustive computation for $n=2d$ and $d \le 5$: 

\begin{conjecture} \label{conj}
Let $\Pi$ be a permutation for which no $t$ exists as in Theorem \ref{thm:US}. Then $\operatorname{USP}\big(\C(V),\pi\big)$ holds for $V$ generic Toeplitz.
\end{conjecture}

On the other hand, for a family of cyclic permutations we have obtained a complete answer:

\begin{proposition}\label{prp:cyclic-full-rk}
Suppose $\pi$ is the cyclic permutation $\pi(i)=i+1$ for $i<n$ and $\pi(n)=1$. Then $\operatorname{USP}\big(\C(V),\pi\big)$ holds for $V$ generic Toeplitz, if and only if $t \ge d/2$.
%$$ \rank \left[ V, \, \Pi^{t} V \right] = \min \{ 2d , d+2|t| \}. $$
\end{proposition}

We next establish a theory that concerns the computation of $v \in \V$ via algebraic means. In \cite{song2018permuted} it was observed that, with $v = V \xi^*$ for $\xi^* \in \KK^d$, $y = \Pi v$, and $x=x_1,\dots,x_d$ variables organized in a column vector, $\xi^*$ is a root of the polynomial $p_\ell(Vx)-p_\ell(y),$ where $\ell$ is any positive integer and  
$$p_\ell(z) :=z_1^\ell+\cdots+z_n^\ell$$ is the $\ell$-th power-sum polynomial in variables $z:=z_1,\dots,z_n$. The following theorem extends the theory of \cite{tsakiris2020algebraic} to Toeplitz matrices:

\begin{theorem} \label{thm:d-polynomials}
With $n \ge d$, $V$ generic $n \times d$ Toeplitz, $\tilde{y} \in \KK^n$, and $\overline{\KK}$ the algebraic closure of $\KK$, the polynomial system of $d$ equations in $d$ unknowns given by 
\begin{align}
p_\ell(Vx) = p_\ell(\tilde{y}), \, \, \, \ell=1,\dots,d \label{eq:d-polynomials}
\end{align} has at most $d!$ roots in $\overline{\KK}^d$. 
\end{theorem}

\noindent With $y = \tilde{y}$, Theorem \ref{thm:d-polynomials} guarantees that the search for $\xi^*$ can be narrowed down to a finite set of vectors, not exceeding $d!$ in number. Moreover, if $\tilde{y}$ is a perturbed version of $y$, as is common in applications, then the square polynomial system \eqref{eq:d-polynomials} is still guaranteed to be consistent with at most $d!$ roots. Theorem \ref{thm:d-polynomials} thus serves as the theoretical justification for the adoption of the algorithm of \cite{tsakiris2020algebraic} for computing $v$ in the Toeplitz case; this we have already demonstrated with a signal processing simulation in \cite{ICASSP26}. 

Finally, consider the overdetermined polynomial system 
\begin{align}
p_\ell(Vx) = p_\ell(y), \, \, \, \ell=1,\dots,d+1 \label{eq:(d+1)-polynomials}
\end{align} which was already studied experimentally and theoretically in \cite{song2018permuted,liang2024ISSAC,Liang2025field} for general non-Toeplitz matrices $V$. We show that the main theoretical finding of \cite{Liang2025field} extends to the Toeplitz case: 

\begin{theorem} \label{thm:(d+1)-polynomials}
With $n \ge d+1$, $V$ generic $n \times d$ Toeplitz matrix and $\xi^*$ generic $d \times 1$ vector, the polynomial system \eqref{eq:(d+1)-polynomials} admits $\xi^*$ as its unique solution.
\end{theorem}

Most of the effort in the rest of the manuscript concentrates on the proof of an auxiliary result, Theorem \ref{thm:MainTechnical} (\S \ref{section:Proof-Auxiliary}). With Theorem \ref{thm:MainTechnical} at hand, Theorem \ref{thm:US} and Proposition \ref{prp:cyclic-full-rk} follow readily. The proof of Theorem \ref{thm:d-polynomials} involves an argument which  can be seen as an alternative proof of Theorem 2 in \cite{tsakiris2020algebraic}, while the rest of its proof is identical to the proof of Theorem 3 in \cite{tsakiris2020algebraic}. Finally, the proof of Theorem \ref{thm:(d+1)-polynomials} consists of a careful check in the lengthy proof of Theorem 1.1 in \cite{Liang2025field}, that a generic Toeplitz matrix is sufficiently generic.

%%%%%%%%
\section{An Auxiliary Theorem} \label{section:Proof-Auxiliary}

This section is devoted to proving the following key statement:

\begin{theorem}\label{thm:MainTechnical}
If there exists an integer $0\leq t \leq \frac{d}{2}$ such that
\[ r_{t} :=  \rank\big( (\Pi - J^t)_{t+1:n,1:n} \big) \leq d - 2t, \] or if there exists an integer $-\frac{d}{2} \le t < 0$ such that \[ r_{t} :=  \rank\big( (\Pi - (J^\top)^{-t})_{1:n+t,1:n} \big) \leq d + 2t,  \] then $ \rank \left[V,\Pi V\right] = d + r_{t} + 2|t|$, for $V$ generic Toeplitz.
\end{theorem}

We prove the statement corresponding to $0 \le t \le \frac{d}{2}$; the case $-\frac{d}{2} \le t <0$ is treated in a similar fashion. 

Throughout $J_n$ and $J_d$ denote the lower triangular $n \times n$ and $d \times d$, respectively, Jordan blocks associated to eigenvalue zero. Also,  $U=(x_{i-j})$ will be an $n \times d$ Toeplitz matrix, with the $x_{i-j}$'s algebraically independent over $\KK$; we will denote by $\mathbb{L}$ the field generated by the entries of $U$ over $\KK$, and $\rank_{\mathbb{L}}(\cdot)$ will indicate the matrix rank over the field $\mathbb{L}$. 

In what follows, we will prove that $\rank_{\mathbb{L}} [U, \, \Pi U] = d+t+r_t$. This will ensure the existence of a $(d+t+r_t)$-minor of $[U, \,  \Pi U]$ which is a non-zero element of $\mathbb{L}$. Write this element as $f/g$, where $f$ and $g$ are polynomials in the entries of $U$ with coefficients in $\KK$. We may view $f$ and $g$ as polynomial functions on the affine space $\mathbb{T}_{n \times d}$ of $n \times d$ Toeplitz matrices with entries in $\KK$. As $\KK$ is infinite and the product $fg$ is non-zero, the polynomial $fg$ defines a hypersurface $\mathscr{Y} \subsetneq \mathbb{T}_{n \times d}$. By construction, for any Toeplitz matrix $V \in \mathbb{T}_{n \times d} \setminus \mathscr{Y}$, the rank of the matrix $[V, \, \Pi V]$ is $d+t+r_t$, which is the conclusion of the theorem.

The first step in the strategy of the proof is to exploit the fact that the matrix $[U, \, \Pi U - UJ_d^t]$ is obtained from the matrix $[U, \, \Pi U]$ by elementary column operations, whence  
\begin{align}
\rank_{\mathbb{L}} [U, \,  \Pi U] = \rank_{\mathbb{L}} [U, \, \Pi U - UJ_d^t]. \label{eq:ElemColOper}
\end{align} \noindent We will prove that $\rank_{\mathbb{L}} [U, \, \Pi U - UJ_d^t] = d+t+r_t$.

The second step in the strategy of the proof is write $[U, \, \Pi U - UJ_d^t]$ as the sum of two suitable matrices $A + B$ and show that $A,B$ satisfy the hypotheses of the following classical result, together with $\rank(A)+\rank(B) = d+t+r_t$: 

\begin{lemma}[\cite{MarsagliaStyan1972}]\label{rk_concat}
Let $A$ and $B$ be matrices of the same size with entries in some field. Then
$\rank (A + B) = \rank (A) + \rank (B)$ if and only if
\begin{align*}
     \rank \left(\left[A, \, B\right]\right) = \rank (A) + \rank (B) = \rank \left(\begin{bmatrix}
        A \\
        B
    \end{bmatrix}\right). 
\end{align*} 
\end{lemma}

\noindent The matrices $A$ and $B$, as well as certain other auxiliary matrices are defined as follows. 
\begin{align*}
X & = \Pi U - J_{n}^t U \\
Y &= J_{n}^t U - U J_{d}^t \\
X_1 &=X_{1:t,1:d-t} \\
X_2 & =X_{1:t,d-t+1:d} \\
X_3 &=X_{t+1:n,1:d-t} \\
X_4 &=X_{t+1:n,d-t+1:d} \\
A &= \begin{bmatrix}
    U & \begin{matrix}
        X_1 + Y_1 & X_2 \\
        0_{t \times (d-t)} & Y_4
    \end{matrix}
\end{bmatrix} \\
B &= \begin{bmatrix}
    0_{n \times d} & \begin{matrix}
        0_{t \times d} \\
        X_{t+1:n,1:d}
    \end{matrix}
\end{bmatrix},
\end{align*}
where $Y_1,Y_2,Y_3,Y_4$ are defined analogously. Observe that 
\begin{align}
Y_1 &=-U_{1:t,t+1:d} \\
Y_2 & =0, \\
Y_3 & = 0, \\
Y_4 & =U_{1:n-t,d-t+1:d}, \label{eq:Y4} \\
A+B & = [U, \, \Pi U - UJ_d^t].
\end{align}

The rest of the proof consists of showing that the hypotheses of Lemma \ref{rk_concat} hold for $A$ and $B$ and that $\rank(A)+\rank(B) = d+t+r_t$.

We begin with the following key lemma:

\begin{lemma}\label{minors_li}
Let $T = (y_{i-j}:i,j\in \mathbb{Z})$ be a Toeplitz matrix of infinite size, with the $y_{i-j}$'s algebraically independent over $\KK$. Then for any submatrix $S$ of size $\ell \times d$ with $\ell \geq d$, the $d$-minors of $S$ are linearly independent over $\KK$.
\end{lemma}
\begin{proof}
Let $S$ be the submatrix of $T$ associated to row indices $\K=\{k_1<\cdots <k_l\} \subset \mathbb{Z}$ and column indices $\J=\{j_1<\cdots <j_d\}\subset  \mathbb{Z}$; we write $S=T_{\K\times \J}$. Let $\I=\{i_1<\cdots <i_d\} \subseteq \K$ and consider the $d$-minor of $S$ given by $\det(T_{\I \times \J})$. Under the lexicographic term order $\succ$ induced by $y_\gamma>y_\delta$ for every $\gamma>\delta$, direct inspection of the matrix 
\begin{align*}
T_{\I,\J} = \begin{bmatrix}
y_{i_1-j_1} & y_{i_1-j_2} & \cdots & y_{i_1-j_d} \\
\vdots & \vdots & \vdots & \vdots \\
y_{i_{d-1}-j_1} & y_{i_{d-1}-j_2} & \cdots & y_{i_{d-1}-j_d} \\
y_{i_{d}-j_1} & y_{i_{d}-j_2} & \cdots & y_{i_{d}-j_d}
\end{bmatrix}
\end{align*} shows that the initial monomial of $\det(T_{\I \times \J})$ is the product of the monomials along the anti-diagonal of $T_{\I,\J}$, i.e.
\begin{align*}
\operatorname{in}_{\succ} \det(T_{\I \times \J})=\prod_{a=1}^{d} y_{i_{a}-j_{d+1-a}}.
\end{align*} Notice that distinct $\I \subseteq \K$ lead to distinct $\operatorname{in}_{\succ} \det(T_{\I \times \J})$. Hence, as $\I$ varies in $\K$, the set of $\operatorname{in}_{\succ} \det(T_{\I \times \J})$'s is a set of distinct monomials of degree $d$, which are thus linearly independent over $\KK$. This certainly implies that the set of $\det(T_{\I \times \J})$'s is linearly independent over $\KK$.
\end{proof}

The next two lemmas will be useful when working with $U$.

\begin{lemma} \label{dim_aug_Toep_matrix}
Let $W$ be an $n \times d_0$ rank-$d_0$ matrix with entries in $\KK$. Then 
$$ \rank_{\mathbb{L}} [W, \, U ] = \min\{d_0+d,n\}.$$
\end{lemma}
\begin{proof}
Set $d_e = \min\{d_0+d,n\}$ and $d_0'=d_e-d>0$. By generalized Laplace expansion, each $d_e$-minor of the $n \times (d_0+d)$ matrix $[W, \, U]$ is a linear combination of $d$-minors of $U$ with coefficients that are up to sign $d_0'$-minors of $W$. As $d_0' \le d_0 = \rank(W)$, there is at least one non-zero $d_0'$-minor of $W$, and there is a $d_e$-minor of $[W, \, U]$ for which the said $d_0'$-minor of $W$ appears as a coefficient in the corresponding linear combination. This linear combination is a non-zero element of $\mathbb{L}$, because by Lemma \ref{minors_li} the $d$-minors of $U$ are linearly independent over $\KK$. 
\end{proof}

\begin{lemma}\label{lem:QU}
With $m \le n$, let $Q$ be any $m \times n$ matrix with entries in $\KK$ and of rank $r \le d$. Then there is an $r \times d$ submatrix of $QU$ for which every $r$-minor is non-zero.
\end{lemma} 
\begin{proof}
Since the rank of $Q$ is $r$, there is an $r \times n$ row-submatrix of $Q$ whose rank is $r$; we may replace $Q$ with that submatrix, i.e. we may assume $m=r$. With any $\J=\{j_1<\cdots<j_r\} \subseteq \{1,\dots,d\}$ and $U_\J$ the column-submatrix of $U$ indexed by $\J$, the Cauchy-Binet formula gives 
$$\det(QU_\J) = \sum_{\K} \det(Q_\K) \det(U_{\K,\J}),$$ where the sum extends over all ordered subsets of $\{1,\dots,n\}$ of cardinality $r$. As $\rank(Q)=r$, there is some $\K'$ such that $\det(Q_{\K'}) \neq 0$. Then Lemma \ref{minors_li} gives $\det(QU_\J) \neq 0$.
\end{proof}

We next determine the ranks of $A$ and $B$. For this, but also later, we will make use of the elementary inequality 
\begin{align}
\rank \begin{pmatrix}
        R & S \\
        0 & T
\end{pmatrix}\geq \rank (R)+ \rank (T). \label{rk_tri}
\end{align}

\begin{lemma} \label{lem:rank-A}
We have $\rank_{\mathbb{L}}(A) = d + 2t$.
\end{lemma}
\begin{proof}
The rank inequality  
\[ \rank_{\mathbb{L}}(A) \leq \rank_{\mathbb{L}}(U)+\rank_{\mathbb{L}}(\underbrace{X_1+Y_1}_{t \times (d-t)}) +\rank_{\mathbb{L}} \underbrace{\begin{bmatrix}X_2 \\ Y_4 \end{bmatrix}}_{n \times d},\] together with $\rank_{\mathbb{L}}(U) = d$ and $t \le \frac{d}{2} \le d-t$, lead to $\rank_{\mathbb{L}}(A) \le d+2t$.

For the reverse direction, \eqref{rk_tri} gives 
\[ \rank_{\mathbb{L}}(A) \geq \rank_{\mathbb{L}}(X_1+Y_1)+ \rank_{\mathbb{L}} 
    [U_{t+1:n,1:d}, \, Y_4].\]
Now, observe that the variable $x_{-d+t}$ does not appear in $X_1$, while it appears in $Y_1$ only in the diagonal $(k,d-2t+k)$ for $k = 1,2,\cdots,t$. It follows that in the expansion of the rightmost $t \times t$ minor of $X_1+Y_1$ the monomial $(x_{-d+t})^t$ appears with a non-zero coefficient in $\KK$, and thus this minor is non-zero; i.e. $\rank_{\mathbb{L}}(X_1+Y_1)=t$. Moreover, by \eqref{eq:Y4}, 
\begin{align*}
[U_{t+1:n,1:d}, \, Y_4] &= [U_{t+1:n,1:d}, \, U_{1:n-t,d-t+1:d}] \\
&= \begin{bmatrix}
\begin{matrix}
x_t & \cdots & x_{-d+t+1} \\
x_{t+1} & \cdots & x_{-d+t+2} \\
\vdots & \cdots & \vdots \\
x_{n-1} & \cdots & x_{n-d} 
\end{matrix} & 
\begin{matrix}
x_{-d+t} & \cdots & x_{1-d} \\
x_{-d+t+1} & \cdots & x_{1-d+1} \\
\vdots & \cdots & \vdots \\
x_{n-d-1}& \cdots & x_{n-d-t}
\end{matrix}
\end{bmatrix}
\end{align*} is an $(n-t) \times (d+t)$ Toeplitz matrix, with algebraically independent over $\KK$ elements along its diagonals. As such, and since $t \le \frac{d}{2}$ implies $d+t \le n-t$, the rank of this Toeplitz matrix is $d+t$ by Lemma \ref{minors_li}. It follows that $\rank_{\mathbb{L}}(A) \ge d+2t$.
\end{proof}
    
\begin{lemma}\label{lem:rank-B}
We have $\rank_{\mathbb{L}}(B) = r_t$.
\end{lemma}  
\begin{proof}
By the definition of $B$, we have $\rank_{\mathbb{L}}(B) = \rank_{\mathbb{L}}(X_{t+1:n,1:d}).$ By the definition of $X$, we have $X_{t+1:n,1:d} = QU$, where $Q$ is the submatrix of $P-J_n^t$ composed by its last $n-t$ rows. By definition, $r_t = \rank_{\mathbb{L}}(Q)$. By hypothesis, $r_t \le d-2t$. Thus by Lemma \ref{lem:QU}, $\rank_{\mathbb{L}}(QU) = r_t$.
\end{proof}  

Lemmas \ref{lem:rank-A} and \ref{lem:rank-B} imply 
$$\rank_{\mathbb{L}}(A) + \rank_{\mathbb{L}}(B) = d+t+r_t.$$ \noindent It remains to prove that this value is also the rank of the horizontal and vertical concatenation of $A$ and $B$; this is done in the remaining two lemmas.

\begin{lemma} \label{lem:rank-horizontal}
We have $\rank_{\mathbb{L}} [A, \, B] = d+2t+r_t$.
\end{lemma}
\begin{proof}
By \eqref{rk_tri} and the definitions of $A$ and $B$,
\begin{align*}
\rank_{\mathbb{L}} [A, \, B] &= \rank_{\mathbb{L}} \hconcat \\
    &\geq  \rank_{\mathbb{L}} (X_1+Y_1) 
    +  \rank_{\mathbb{L}} 
            [U_{t+1:n,1:d}, \, Y_4, \, X_3, \, X_4].
\end{align*} As in the proof of Lemma \ref{lem:rank-A}, the rank of $X_1+Y_1$ is $t$. We view the second matrix as the horizontal concatenation of the two matrices $[U_{t+1:n,1:d}, \, Y_4]$ and $[X_3, \, X_4]$. By Lemma \ref{lem:rank-B} the rank of $[X_3, \, X_4]$ is $r_t$. By the proof of Lemma \ref{lem:rank-B}, there exists an invertible $(n-t) \times (n-t)$ matrix $Q$ with entries in $\KK$, such that only the last $r_t$ rows of $Q(P-J_n^t)_{t+1:n,1:n}$ are non-zero; thus the same will hold true for the matrix $[QX_3, \, Q X_4]=Q(P-J_n^t)_{t+1:n,1:n}U$. As noted in the proof of Lemma \ref{lem:rank-A}, the matrix $[U_{t+1:n,1:d}, \, Y_4]$ is Toeplitz of size $(n-t) \times (d+t)$, so that the matrix $[Q U_{t+1:n,1:d}, \, QY_4, \, Q X_3, \, Q X_4]$ is block lower triangular, with the the top left block being of size $(n-t-r_t) \times (d+t)$; denote that block by $H$. Now, the hypothesis asserts that $r_t \le d-2t$; as $n \ge 2d$, we have $r_t \le n-d-2t$ or equivalently, $d+t \le n-t-r_t$. In other words, the set of $(d+t)$-minors of $H$ is non-empty. By the functorial nature of the exterior power, the set of $(d+t)$-minors of $Q [V_{t+1:n,1:d}, \, Y_4]$, viewed as a vector, is obtained from the set of $(d+t)$-minors of $[V_{t+1:n,1:d}, \, Y_4]$ via multiplication by an invertible matrix with entries in $\KK$. As the set of $(d+t)$-minors of $[V_{t+1:n,1:d}, \, Y_4]$ is linearly independent over $\KK$ by Lemma \ref{minors_li}, the same will be true for the $(d+t)$-minors of $Q[V_{t+1:n,1:d}, \, Y_4]$; in particular, every $(d+t)$-minor of $Q[V_{t+1:n,1:d}, \, Y_4]$ will be non-zero. It follows that there is a non-zero $(d+t)$-minor of $H$, so that $$\rank_{\mathbb{L}}[U_{t+1:n,1:d}, \, Y_4, \, X_3, \, X_4] \ge d+t+r_t.$$ In fact, since $[U_{t+1:n,1:d}, \, Y_4]$ has size $(n-t) \times (d+t)$, the above inequality is equality. We have shown that $$\rank_{\mathbb{L}} [A, \, B] \ge d+2t+r_t,$$ and since by Lemmas \ref{lem:rank-A} and \ref{lem:rank-B} $$d+2t+r_t = \rank_{\mathbb{L}}(A)+\rank_{\mathbb{L}}(B) \ge \rank_{\mathbb{L}}[A, \, B],$$ equality must hold.
\end{proof}

\begin{lemma} \label{lem:rank-vertical}
We have $\rank_{\mathbb{L}} \begin{bmatrix} A \\ B \end{bmatrix} = d+2t+r_t$.
\end{lemma}
\begin{proof}
By \eqref{rk_tri}, we have the inequality
\begin{align*}
\rank_{\mathbb{L}} \begin{bmatrix} A \\ B \end{bmatrix} &= \rank_{\mathbb{L}} \vconcat  \\
& \geq \rank_{\mathbb{L}} \begin{bmatrix}
    X_1+Y_1 \\ X_3
\end{bmatrix} + \underbrace{\rank_{\mathbb{L}}[U_{t+1:n,1:d}, \, Y_4]}_{=d+t}.
\end{align*} Now observe that the variable $x_{-t}$ appears only in the diagonal $(k, k)$ of $Y_1$ for $k=1,\dots,t$ and does not appear at all in $X_1$ or $X_3$. By Lemma \ref{lem:QU}, there is an $r_t \times (d-t)$ row-submatrix of $X_3$ for which every $r_t$-minor is non-zero. Let $\I = \{i_1<\cdots <i_{r_t}\}$ be the row indices indexing this submatrix. Then 
\begin{align*}
\det \left( 
\begin{bmatrix}
    X_1+Y_1 \\ X_3
\end{bmatrix}_{\{1...t\}\cup \I, \{1...t\} \cup \{t+1,\dots,t+r_t\}}
\right) \neq 0
\end{align*} because the monomial $x_{-t}^t$ appears (up to sign) with coefficient a non-zero $r_t$-minor of $X_3$. As the rank of $X_1+Y_1$ is $t$ and the rank of $X_3$ is $r_t$, this establishes 
$$ \rank_{\mathbb{L}} \begin{bmatrix}
    X_1+Y_1 \\ X_3
\end{bmatrix} = t+r_t.$$ With this, we conclude as in the proof of Lemma \ref{lem:rank-horizontal}. 
\end{proof}

%%%%%%%%%%%%%
\section{Remaining Proofs} \label{section:Rest-Proofs}

\subsection{Proof of Theorem \ref{thm:US}} \label{subsection:Proof-thm:US}

Observe that the property $\operatorname{USP}(\V,\pi)$ is equivalent to saying that $\V$ meets its image $\pi(\V)$ only inside the eigenspace of $\pi$ associated to eigenvalue $1$; i.e. 
\begin{align} 
\operatorname{USP}(\V,\pi) \, \, \, \Leftrightarrow \, \, \,  \V \cap \pi(\V)  = \mathcal{N}(I-\Pi) \cap \V; \label{eq:USP-equivalence}
\end{align} here $I$ is the $n \times n$ identity matrix, $\Pi$ is the permutation matrix that represents $\pi$, and $\mathcal{N}(I-\Pi)$ is the right nullspace of the matrix $I -\Pi$. Set $r_0:= \rank(I-\Pi)$; this is the codimension of the eigenspace of $\pi$ associated to eigenvalue $1$, that is $r_0 = \codim \mathcal{N}(I-\Pi)$. By Lemma \ref{dim_aug_Toep_matrix} and for $V$ generic Toeplitz, the vector space $\mathcal{N}(I-\Pi) \cap \V$ has the expected dimension, i.e. 
$$\dim \big(\mathcal{N}(I-\Pi) \cap \V \big) = \max\{d-r_0,0\}.$$ Hence, as there is always an inclusion of vector spaces 
$$\mathcal{N}(I-\Pi) \cap \V \subseteq \V \cap \pi(\V),$$ and after expressing via Grassman's formula the dimension of $\V \cap \pi(\V)$ in terms of the dimension of $\V + \pi(\V) = \mathcal{C}([V, \, \Pi V])$, where $\C([V, \, \Pi V])$ indicates the column-space of the $n \times 2d$ matrix $[V, \, \Pi V]$, the conditions in \eqref{eq:USP-equivalence} become equivalent to   
\begin{align}
\rank [V, \, \Pi V] =2d - \max\{d-r_0,0\}. \label{eq:USP-equivalence-matrix}
\end{align}
%If $r_0\leq d$, then for generic Toeplitz subspaces we have $\dim \ker(I-P)\cap \mathcal{X} = d - r_0$ (as we will see in lemma \ref{dim_aug_Toep}),
%so $\textup{hsp}(\mathcal{X},\{I,P\})$ is equivalent to a statement on matrices: $ \rk(V,PV) = d + r_0$ for generic Toeplitz matrices $V$ of size $n\times d$;
%if $r_0>d$, then for generic Toeplitz subspaces we have $\dim \ker(I-P)\cap \mathcal{X} = 0$,
%so $\textup{hsp}(\mathcal{X},\{I,P\})$ is equivalent to a statement on matrices: $ \rk(V,PV) = 2d$ for generic Toeplitz matrices $V$ of size $n\times d$.

We prove that \eqref{eq:USP-equivalence-matrix} holds true. Note that the condition $$r_0:=\rank(\Pi-I)\leq d,$$ is equivalent to the hypothesis of Theorem \ref{thm:MainTechnical} for $t=0$. Thus if $r_0 \le d$, then by Theorem \ref{thm:MainTechnical} we have $$\rank [V, \, \Pi V] = d+r_0 = 2d-\max\{d-r_0,0\}.$$ 

If on the other hand $r_0 > d$, we must show that under the hypotheses of the statement, we have $\rank [V, \, \Pi V] = 2d$. Suppose thus there is a $0 \le t \le \frac{d}{2}$ such that $$r_t:=\rank(\Pi - J^{t})_{t+1:n,1:n} \le d-2t$$ \noindent (the case of negative $t$ is treated similarly). Then Theorem \ref{thm:MainTechnical} asserts that $$\rank [V, \, \Pi V] = d+r_t+2t,$$ and this rank is equal to $2d$ if and only if $r_t = d-2t$.

%%%%%%%%%%%%%
\subsection{Proof of Proposition \ref{prp:cyclic-full-rk}}

Note that for a cyclic permutation $r_0 = n-1$. Thus by \eqref{eq:USP-equivalence-matrix}, and since by hypothesis $n \ge 2d$, we have that $\operatorname{USP}\big(\C(V),\pi^t\big)$ holds if and only if $\rank \left[ V, \, P^{t} V \right] = 2d$.

Note that $(\Pi^t - J^{t})_{t+1:n,1:n}$ is the zero matrix. Thus if $t \le \frac{d}{2}$, the hypothesis of Theorem \ref{thm:MainTechnical} holds true with $r_t = 0$, whence $\rank [V, \, \Pi V] = d+2t$. Hence $\operatorname{USP}\big(\C(V),\pi^t\big)$ fails for $t < \frac{d}{2}$ but holds for $t = \frac{d}{2}$.

To finish the proof, we will prove that for $t>\frac{d}{2}$, the maximal minor of $[U, \, \Pi^tU]$ corresponding to the first $2d$ rows is non-zero. Refer to the following partition: 
\begin{equation*}
    [U, \, \Pi^t U] = 
\begin{tikzpicture}[baseline={-0.5ex},mymatrixenv]
    \matrix [mymatrix,inner sep=4pt] (m)  
    {
            *        &U_{1:t,d-t+1:d}&       *          &     *       \\
            *        &     *       &       *          &U_{1:t,d-t+1:d}\\
    U_{2t+1:2d,1:d-t}&     *       &U_{t+1:2d-t,1:d-t}&     *       \\
            *        &     *       &       *          &     *       \\
    };
    % Braces     
    \mymatrixbraceleft{1}{1}{$t$}
    \mymatrixbraceleft{2}{2}{$t$}
    \mymatrixbraceleft{3}{3}{$2d-2t$}
    \mymatrixbraceleft{4}{4}{$n-2d$}
    \mymatrixbracetop{1}{1}{$d-t$}
    \mymatrixbracetop{2}{2}{$t$}
    \mymatrixbracetop{3}{3}{$d-t$}
    \mymatrixbracetop{4}{4}{$t$}
    %\mymatrixbracebottom{3}{3}{$F'$}
    %\mymatrixbracebottom{4}{4}{$F''$}
    %\mymatrixbraceleft{3}{3}{$E'$}
    %\mymatrixbraceleft{4}{4}{$E''$}
\end{tikzpicture}
\end{equation*}

\noindent The first key observation is that the variable $x_{t-d}$ occurs only along the diagonal of the two blocks $U_{1:t,d-t+1:d}$ of $[U, \, \Pi^t U]$. Regarding the aforementioned maximal minor as a polynomial of $x_{t-d}$, the coefficient of $x_{t-d}^{2t}$ is $\det(S)$, where $$\underbrace{S}_{(2d-2t)\times (2d-2t)} = [U_{2t+1:2d,1:d-t}, \, U_{t+1:2d-t,1:d-t}].$$ The second key observation is that $$U_{t+1:2d-t,1:d-t} = U_{2t+1:2d,t+1:d},$$ combined with the fact that when $d-t<t+1$, then $S$ is a square submatrix of the Toeplitz matrix $U$, which is thus of non-zero determinant by Lemma \ref{minors_li}.

%%%%%%
\subsection{Proof of Theorem \ref{thm:d-polynomials}}

The proof of Theorem 3 in \cite{tsakiris2020algebraic}, which we will not repeat here, asserts that for any $\tilde{y} \in \KK^n$, the polynomial system of equations \begin{align*}
p_\ell(Ax) = p_\ell(\tilde{y}), \, \, \, \, \, \,  \ell=1,\dots,d 
\end{align*} is consistent with at most $d!$ solutions in $\overline{\KK}$, for any $n \times d$ matrix $A$ (not necessarily Toeplitz) with entries in\footnote{The statement in \cite{tsakiris2020algebraic} is in fact is over $\Re$, but the proof is valid for any infinite field $\KK$.} $\KK$, as long as the polynomials 
\begin{align}
p_1(Ax),\dots,p_d(Ax) \label{eq:p-regular}
\end{align} form a regular sequence \cite{Eisenbud} (see also Appendix D in \cite{tsakiris2020algebraic}). Thus by the same arguments as in \cite{tsakiris2020algebraic}, the same statement holds true with $A$ replaced by a Toeplitz matrix $V$, as soon as the polynomials 
\begin{align}
p_1(Vx),\dots,p_d(Vx) \label{eq:p-V-regular}
\end{align} form a regular sequence. 

Note that these polynomials are homogeneous and parametrized by the matrix $V$. Each choice of $V$ is a choice of parameters; thus the parameter space is the affine space $\mathbb{T}_{n \times d}$ of $n \times d$ Toeplitz matrices with entries in $\KK$. Now, quite generally, it is well known that for homogeneous polynomials parametrized by an affine space, the polynomials are a regular sequence in an open subset in the Zariski topology of the parameter space\footnote{For a justification, see the first two paragraphs of the proof of Lemma 3.2 in \cite{Liang2025field}.}. Hence there is an open set $\mathscr{U} \subseteq \mathbb{T}_{n \times d}$, such that for any $V \in \mathscr{U}$ the conclusion of Theorem \ref{thm:d-polynomials} is true. To finish the proof, we must show that $\mathscr{U}$ is not empty; for then it will be dense in $\mathbb{T}_{n \times d}$, as it follows from the elementary properties of the Zariski topology.

Let $V^*$ be the $n \times d$ Toeplitz matrix that consists of the $d \times d$ identity matrix on the top $d \times d$ block, while the bottom $(n-d) \times d$ block is the zero matrix. Then 
$$p_\ell(V^* x) = x_1^\ell + \cdots + x_d^\ell,$$
so that $p_1(V^* x),\dots,p_d(V^* x)$ become the first $d$ power-sum polynomials in variables $x_1,\dots,x_d$. That these are a regular sequence, is well-known (e.g., \cite[Lemma 2]{tsakiris2020algebraic}).

%%%%%%
\subsection{Proof of Theorem \ref{thm:(d+1)-polynomials}}

Theorem 1.1 in \cite{Liang2025field} asserts that, with $n \ge d+1$, if $A$ is a generic $n \times d$ matrix with entries in $\KK$, and $\xi^*$ is a generic $d \times 1$ vector with entries in $\KK$, then the system of polynomial equations
\begin{align}
p_\ell(Ax) = p_\ell(\tilde{y}), \, \, \, \ell=1,\dots,d+1 \label{eq:(d+1)-polynomials-A}
\end{align} has $\xi^*$ as its unique solution. More precisely, with $\mathbb{M}_{n \times d}$ the affine space of $n \times d$ matrices with entries in $\KK$, there exists a dense open set $\mathscr{U}$ in the Zariski topology of $\mathbb{M}_{n \times d}$, such that the said statement holds true for every $A \in \mathscr{U}$. The proof of Theorem 1.1 in \cite{Liang2025field} is lengthy; dissecting it, reveals that the same statement can be proved with $A$ replaced by a generic Toeplitz $V$, as soon as the certain three properties are shown to hold true. In the remaining of the proof here we list and prove these properties. 

First, all maximal minors of an $n \times d$ Toeplitz matrix $V$ with its diagonals defined by algebraically independent elements (just as in Lemma \ref{minors_li}) must be non-zero (the analogue of this for the non-Toeplitz case is needed in the proof of Lemma 2.5 in \cite{Liang2025field}). This is indeed true by virtue of Lemma \ref{minors_li}.

Second, with $V$ as in the previous paragraph, the polynomials $p_\ell(Vx)$ for $\ell=1,\dots,d$ must be a regular sequence in the polynomial ring $\KK(V)[x_1,\dots,x_d]$, where $\KK(V)$ is the field of rational functions in the variables that define the diagonals of $V$ (the analogue of this for the non-Toeplitz case is needed in multiple places in the proof of Theorem 1.1 in \cite{Liang2025field}). This follows from the same argument as in the proof of Lemma 3.2 in \cite{Liang2025field}, together with what we showed in the proof of Theorem \ref{thm:(d+1)-polynomials}, i.e. that these polynomials are a regular sequence upon substituting $V$ with the special matrix whose top $d \times d$ block is the identity and all remaining entries are zero.

Third, and final requirement, is that the codimension of the algebraic variety of $\mathbb{T}_{(d+1) \times d}$ defined by the maximal minors of a $(d+1) \times d$ Toeplitz matrix of variables is at least $2$ (the analogue of this for the non-Toeplitz case is needed in the proof of Lemma 4.8 in \cite{Liang2025field}). The codimension of such Toeplitz (or equivalently Hankel) variety is read from \cite[p.133]{Conca2018Hankel} to be indeed equal to $2$.

%%%%%%%%%%%%%
\section{Examples} \label{section:Examples}

\begin{example}\label{eg:main_tech}
With $d=3$ and $n=6$, consider the permutation $\pi$ represented by
    \[ P=\begin{bmatrix}
        1 & 0 & 0 & 0 & 0 & 0\\
        0 & 0 & 0 & 0 & 0 & 1\\
        0 & 1 & 0 & 0 & 0 & 0 \\
        0 & 0 & 1 & 0 & 0 & 0 \\
        0 & 0 & 0 & 1 & 0 & 0 \\
        0 & 0 & 0 & 0 & 1 & 0
    \end{bmatrix},\] which consists of two cycles, i.e. a fixed point and a cycle of length $5$. We have 
$$ (P - J)_{2:6,1:6} =    
\begin{bmatrix}        
        1 & 0 & 0 & 0 & 0 & 1\\
        0 & 0 & 0 & 0 & 0 & 0 \\
        0 & 0 & 0 & 0 & 0 & 0 \\
        0 & 0 & 0 & 0 & 0 & 0 \\
        0 & 0 & 0 & 0 & 0 & 0
    \end{bmatrix};$$ hence $r_1 = 1$. Thus with $t=1$, $1 = r_t \le d - 2t = 1$; i.e. the hypothesis of Theorem \ref{thm:MainTechnical} holds true for $t=1$. It follows that $\rank [V, \, PV] = d+2t+r_t = 6$ for generic $6 \times 3$ Toeplitz $V$. As $r_t = d-2t$, we have in fact that $\operatorname{USP}(\C(V),\pi)$ holds true by Theorem \ref{thm:US}.      
\end{example}

\begin{example}\label{eg:full_tk}
With $d=3$ and $n=6$, consider the permutation $\pi$ represented by
$$Q = \begin{bmatrix}
0 & 0 & 0 & 1 & 0 & 0\\
0 & 0 & 0 & 0 & 0 & 1\\
0 & 1 & 0 & 0 & 0 & 0 \\
1 & 0 & 0 & 0 & 0 & 0 \\
0 & 0 & 1 & 0 & 0 & 0 \\
0 & 0 & 0 & 0 & 1 & 0
\end{bmatrix}.$$ We have $Q = P^2$, where
$$ P = \begin{bmatrix}
0 & 0 & 0 & 0 & 0 & 1\\
1 & 0 & 0 & 0 & 0 & 0\\
0 & 1 & 0 & 0 & 0 & 0 \\
0 & 0 & 1 & 0 & 0 & 0 \\
0 & 0 & 0 & 1 & 0 & 0 \\
0 & 0 & 0 & 0 & 1 & 0
\end{bmatrix}.$$ Note that $2 = t > \frac{d}{2} = \frac{3}{2}$, so that the hypothesis of Theorem \ref{thm:MainTechnical} is not valid. On the other hand, Proposition \ref{prp:cyclic-full-rk} asserts that $\rank [V, \, QV] = \min\{2d,d+2t\} = 6$ for a $6 \times 3$ generic Toeplitz $V$. In fact, a by hand computation reveals that $$\det [U, \, QU] =  x_5^2 x_3 x_2 x_{1} x_{-1}$$ (here $U$ is a $6 \times 3$ Toeplitz matrix of variables; see the proof of Theorem \ref{thm:MainTechnical} for the definition).  The remark after Proposition \ref{prp:cyclic-full-rk}, also shows that $\operatorname{USP}(\C(V),\pi)$ holds true.
\end{example}

\begin{example}\label{eg:not_covered}
Consider the $3$-cycle 
$$ Q = \begin{bmatrix}
0 & 0 & 1 \\
1 & 0 & 0 \\
0 & 1 & 0
\end{bmatrix}, $$ and with $d=3$ and $n=6$ the permutation $\pi$ represented by 
$$ P = \begin{bmatrix}
Q^2 & 0 \\
0 & Q 
\end{bmatrix} = 
\begin{bmatrix}
0 & 1 & 0 & 0 & 0 & 0\\
0 & 0 & 1 & 0 & 0 & 0\\
1 & 0 & 0 & 0 & 0 & 0 \\
0 & 0 & 0 & 0 & 0 & 1 \\
0 & 0 & 0 & 1 & 0 & 0 \\
0 & 0 & 0 & 0 & 1 & 0
\end{bmatrix}.$$
One checks that $\pi$ does not adhere to the hypotheses of either Theorem \ref{thm:MainTechnical} or Proposition \ref{prp:cyclic-full-rk}. Nevertheless, a by-hand computation reveals that \[ \det[U, \, PU] =  x_5^2 x_3 x_1 x_{0} x_{-1},\]
thus illustrating Conjecture \ref{conj}.
\end{example}

\bibliographystyle{alpha}
\bibliography{Hong-Tsakiris-26}

\end{document}